\documentclass[a4paper,12pt]{amsart}
\usepackage{latexsym,amssymb}

\DeclareMathOperator{\Aut}{Aut}

\newtheorem{theorem}{Theorem}%[section]

\newtheorem{lemma}[theorem]{Lemma}
\newtheorem{corollary}[theorem]{Corollary}

\theoremstyle{definition}

\newtheorem*{remark}{Remark}

\DeclareMathOperator{\PgL}{P\Gamma L}
\newcommand{\abs}[1]{\lvert#1\rvert}

\newcommand{\ZZZ}{\mathbb Z}

\newcommand{\FFF}{\mathbb F}

\title[On the non--existence of sharply transitive sets]{On the non--existence
  of sharply transitive sets of permutations in certain finite permutation
  groups}

\author{Peter M\"uller} 
\email{peter.mueller@mathematik.uni-wuerzburg.de}
\author{G\'abor P.\ Nagy} 
\email{nagyg@math.u-szeged.hu}

\address{Institut f\"ur Mathematik, Universit\"at W\"urzburg, Am
Hubland, D-97074 W\"urz\-burg, Germany}
\address{Bolyai Institute, University of Szeged, Aradi v\'ertan\'uk
tere 1, H-6720 Sze\-ged, Hungary}

\thanks{The second author was supported by DAAD and TAMOP project
4.2.2-08/1/2008-0008.}

\begin{document}

\begin{abstract}
In this short note we present a simple combinatorial trick which can be
effectively applied to show the non-existence of sharply transitive sets of
permutations in certain finite permutation groups.
\end{abstract}

\maketitle 

\section{Introduction}

A permutation code (or array) of length $n$ and distance $d$ is a set $S$ of
permutations of some fixed set $\Omega$ of $n$ symbols such that the Hamming
distance between each distinct $x,y \in S$ is at least $d$, see
\cite{FrankDeza}. By elementary counting, one has $|S|\leq n(n-1)\cdots d$ and
equality holds if and only if $S$ for any two tuples $(x_1,\ldots,x_{n-d+1})$,
$(y_1,\ldots,y_{n-d+1})$ of distinct symbols, there is a unique element $s\in
S$ with $x_1^s=y_1,\ldots,x_{n-d+1}^s=y_{n-d+1}$. Such sets of permutations
are called \emph{sharply $t$-transitive}, where $t=n-d+1$. It is well known
that sharply $1$- and $2$-transitive sets of permutations correspond to Latin
squares and affine planes, respectively \cite{Dembowski}.

In general, there are very few results on permutation codes and there is a
large gap between the lower and upper estimates for $|S|$; see \cite{Tarnanen},
\cite{Quistorff}. Most of the known constructions are related to multiply
transitive permutation groups. In the 1970's, P. Lorimer started the systematic
investigation of the question of existence of sharply $2$-transitive sets
in finite $2$-transitive permutation groups. This program was continued by Th.
Grundh\"ofer, M. E. O'Nan, P. M\"uller, see \cite{GrundhoeferMueller} and the
references therein. Some of the $2$-transitive permutation groups needed rather
elaborated methods from character theory in order to show that they do not
contain sharply $2$-transitive sets of permutations. 

In this paper, we present some simple combinatorial methods which are useful to
exclude the existence of sharply $1$- and $2$-transitive sets of permutations
in given finite permutation groups. 

Notice that if $S$ is a sharply $t$-transitive set of permutations on $\Omega$,
then it is also a sharply $1$-transitive set of permutations on the set
$\Omega^{(t)}$ of $t$-arrangements of $\Omega$. In other words, the
$t$-transitive permutation group $G$ contains a sharply $t$-transitive set if
and only if in its induced action on $\Omega^{(t)}$, $G$ contains a sharply
$1$-transitive set. 

Let $G$ be a permutation group on the set $\Omega=\{\omega_1,\ldots,\omega_n\}$
and for $g\in G$, denote by $\pi(g)$ the corresponding permutation matrix. Let
$J$ denote the $n\times n$ all-one matrix. The existence of sharply
transitive sets in $G$ is equivalent to the $\{0,1\}$-solvability of the matrix
equation
\begin{equation}\label{eq:01}
\sum_{g\in G} x_g \pi(g) = J.
\end{equation}
For some permutation groups we are able to show that \eqref{eq:01} has no
integer solution, which implies the nonexistence of a sharply transitive set in
the given group. 

% Let $\Omega^{(m)}$ be the set of $m$--tuples with pairwise distinct components
% from $\Omega$. We say that a set $S$ of permutations of $\Omega$ is
% \emph{sharply $m$-transitive on $\Omega$} if for each $(x_1,\ldots,x_m)$,
% $(y_1,\ldots,y_m)\in\Omega^{(m)}$ there is a unique $g\in S$ such that
% $x_i^g=y_i$ for all $i$. Note that $S$ is sharply $m$--transitive on $\Omega$
% if and only if $S$ is sharply transitive on $\Omega^{(m)}$.

\section{Contradicting subsets} 

The following simple lemma will be our main tool. 
\begin{lemma} \label{lm:doublecount}
Let $S$ be a sharply transitive set of permutations on a finite set $\Omega$. 
Let $B$ and $C$ be arbitrary subsets of $\Omega$. Then $\sum_{g\in S}\abs{B\cap
C^g}=\abs{B}\abs{C}$.
\end{lemma}
\begin{proof} Count the set of triples $(b,c,g)$, where $b\in B$, $c\in C$,  
$g\in S$ and $c^g=b$, in two ways: If $b,c$ is given, then there is a unique
$g$ by sharp transitivity. If $g$ is given, then the number of pairs $b,c$ is
$\abs{B\cap C^g}$.
\end{proof}

An immediate consequence is
\begin{lemma} \label{lm:noshrtrset}
Let $G$ be a permutation group on a finite set $\Omega$. Assume that there are
subsets $B$, $C$ of $\Omega$ and a prime $p$ such that $p\nmid\abs{B}\abs{C}$
and $p\mid\abs{B\cap C^g}$ for all $g\in G$. Then $G$ contains no sharply
transitive set of permutations.
\end{lemma}
\begin{remark}
It is easy to see that under the assumption of Lemma \ref{lm:noshrtrset}, the
system \eqref{eq:01} does not have a solution in the finite field $\FFF_p$, so
in particular \eqref{eq:01} has no integral solution.
\end{remark}

We give several applications of these lemmas. First, we show that in even
characteristic, the symplectic group does not contain sharply transitive sets of
permutations.

\begin{theorem} \label{th:sp2n2}
Let $n,m$ be positive integers, $n\geq 2$, $q=2^m$. Let $G_1=PSp(2n,q)\rtimes
\Aut(\mathbb F_q)$ and $G_2=Sp(2n,q) \rtimes \Aut(\mathbb F_q)$ be permutation
groups in their natural permutation actions on $\Omega_1=PG(2n-1,q)$ and
$\Omega_2=\mathbb F_q^{2n}\setminus \{0\}$. Then, $G_1$ and $G_2$ do not
contain a sharply transitive set of permutations.
\end{theorem}
\begin{proof}
We deal first with the projective group $G_1$. Let $\mathcal E$ be an elliptic
quadric whose quadratic equation polarizes to the invariant symplectic form
$\langle .,. \rangle$ of $G_1$. Let $\ell$ be a line of $PG(2n-1,q)$ which is
nonsingular with respect to $\langle .,. \rangle$. Then for any $g\in G_1$,
$\ell^g$ is nonsingular, that is, it is not tangent to $\mathcal E$. In
particular, $|\mathcal E \cap \ell^g|=0$ or $2$ for all $g\in
G_1$. Furthermore, we have
\[|\mathcal E|=\frac{q^{2n-1}-1}{q-1} - q^{n-1}, \hspace{1cm} |\ell|=q+1, \]
both odd for $n\geq 2$. We apply Lemma \ref{lm:noshrtrset} with $B=\mathcal E$,
$C=\ell$ and $p=2$ to obtain the result of the theorem. 

In order to show the result for the group $G_2$, we define the subsets $\mathcal
E'=\varphi^{-1}(\mathcal E)$ and $\ell'=\varphi^{-1}(\ell)$, where
$\varphi:\Omega_2\to \Omega_1$ is the natural surjective map. Then, 
\[|\mathcal E'|=(q-1)|\mathcal E|, |\ell'|=(q-1)|\ell| \mbox{ and } |\mathcal
E' \cap \ell'| \in \{0,2(q-1)\}.\]
Hence, Lemma \ref{lm:noshrtrset} can be applied with $B=\mathcal E'$, $C=\ell'$
and $p=2$.
\end{proof}

It was a long standing open problem wether the Mathieu group $M_{22}$ contains
a sharply transitive set of permutations, cf.\ \cite{Grundhoefer}. The
negative answer given in the following theorem implies the nonexistence of
sharply $2$-transitive sets in the Mathieu group $M_{23}$.

We will use the Witt design $\mathcal W_{23}$. This is a $(23,7,4)$--Steiner
system. The fact which we use here and again in the proof of Theorem
\ref{T:McL} is that any two blocks of $\mathcal W_{23}$ intersect in $1$, $3$,
or $7$ points.
\begin{theorem} \label{th:m22}
In its natural permutation representation of degree $22$, the Mathieu group
$M_{22}$ does not contain a sharply transitive set of permutations.
\end{theorem}
\begin{proof}
Let $\Omega'=\{1,\ldots,23\}$, $\Omega=\{1,\ldots,22\}$ and $G=M_{22}$ be the
stabilizer of $23 \in \Omega'$. Let $B\subset \Omega$ be a block of the Witt
design $\mathcal W_{23}$, and $C=\Omega \setminus B$. Then, $|B|=7, |C|=15$
and for all $g\in G$, $|B\cap C^g| = 0,4 \mbox{ or } 6$. Lemma
\ref{lm:noshrtrset} implies the result with $p=2$.
\end{proof}

We can apply our method for certain alternating groups, as well. The following
simple result is somewhat surprising because until now, the symmetric and
alternating groups seemed to be out of scope in this problem.

\begin{theorem} \label{th:altn}
If $n\equiv 2,3 \pmod{4}$ then the alternating group $A_n$ does not contain a
sharply $2$-transitive set of permutations.
\end{theorem}
\begin{proof}
Assume $n\equiv 2,3 \pmod{4}$ and let $G$ be the permutation action of $A_n$ on
the set $\Omega^{(2)}$ with $\Omega=\{1,\ldots,n\}$. A sharply $2$-transitive
set of permutations in $A_n$ corresponds to a sharply transitive set of
permutations in $G$. Define the subsets
\[B=\{(x,y) \mid x<y\}, \hspace{1cm} C=\{(x,y) \mid x>y\}\]
of $\Omega^{(2)}$. By the assumption on $n$, $|B|=|C|=n(n-1)/2$ is odd. For any
permutation $g\in S_n$, we have
\[|\{(x,y) \mid x<y,\; x^g>y^g\}| \equiv \mathrm{sgn}(g) \pmod{2}.\]
This implies $|B\cap C^g|\equiv 0 \pmod{2}$ for all $g \in A_n$. Thus, we
can apply Lemma \ref{lm:noshrtrset} to obtain the nonexistence of sharply
transitive sets in $G$ and sharply $2$-transitive sets in $A_n$. 
\end{proof}

Theorems \ref{th:m22} and \ref{th:altn} can be used to prove the
nonexistence of sharply $2$-transitive sets in the Mathieu group $M_{23}$. 
\begin{corollary}
In its natural permutation representation of degree $23$, the Mathieu group
$M_{23}$ does not contain a sharply $2$-transitive set of permutations.
\end{corollary}

As the last application of our contradicting subset method, we deal with the
stabilizer of the sporadic group $Co_3$ in its doubly transitive action on $276$
points. As a corollary, we obtain a purely combinatorial proof for a
theorem by Grundh\"ofer and M\"uller saying that $Co_3$ has no sharply
$2$-transitive set of permutations. Notice that the original proof used the
Atlas of Brauer characters.

\begin{theorem}\label{T:McL}
Let $G$ be the group $McL\!:\!2$ in its primitive permutation action on $275$
points. Then, $G$ does not contain a sharply transitive set of permutations. 
\end{theorem}
\begin{proof}
Identify $G$ with the automorphism group of the McLaughlin graph $\Gamma$,
acting on the $275$ vertices. We claim that there are subsets $B$ and $C$ of
vertices with $\abs{B}=22$, $\abs{C}=56$, and $\abs{B\cap C^g}\in\{0,3,6,12\}$
for all $g\in G$. The theorem then follows from Lemma \ref{lm:noshrtrset} with
$p=3$.

In order to describe $B$ and $C$, we use the construction of $\Gamma$ based on
the Witt design $\mathcal W_{23}$, see e.g.\ \cite[11.4.H]{BCN:graphs}. Let
$B\cup\{q\}$ be the $23$ points of $\mathcal W_{23}$. Let $U$ be the $77$
blocks of $\mathcal W_{23}$ which contain $q$, and $V$ be the $176$ blocks
which do not contain $q$. The vertices of $\Gamma$ are the $22+76+176=275$
elements from $B\cup U\cup V$. Adjacency $\sim$ on $\Gamma$ is defined as
follows: The elements in $B$ are pairwise non--adjacent. Furthermore, for
$b\in B$, $u,u'\in U$, $v,v'\in V$ define: $b\sim u$ if $b\not\in u$, $b\sim
v$ if $b\in v$, $u\sim u'$ if $\abs{u\cap u'}=1$ (so $u\cap u'=\{q\}$), $v\sim
v'$ if $\abs{v\cap v'}=1$, and $u\sim v$ if $\abs{u\cap v}=3$.

This construction gives the strongly regular graph $\Gamma$ with parameters
$(275,112,30,56)$. Pick two vertices $i\ne j$ which are not adjacent, and let
$C$ be the set of vertices which are adjacent to $i$ and $j$. Then
$\abs{C}=56$. For $g\in G=\text{Aut}(\Gamma)$, $C^g$ is again the common
neighborhood of two non--adjacent vertices. Thus without loss of generality we
may assume $g=1$, so we need to show that $\abs{B\cap C}=0,3,6$, or
$12$. Suppose that $\abs{B\cap C}>0$. Then there is a vertex $x\in B\cap C$
which is adjacent to $i$ and $j$. Therefore $i,j\not\in B$. Recall that two
distinct blocks of $\mathcal W_{23}$ intersect in either $1$ or $3$ points.

We have to consider three cases: First $i,j\in U$. Then $\abs{i\cap j}=3$ and
$q\in i\cap j$. Furthermore, $B\cap C=B\setminus(i\cup j)$, so $\abs{B\cap
  C}=12$. Next, if $i,j\in V$, then $\abs{i\cap j}=3$ and $B\cap C=i\cap j$,
so $\abs{B\cap C}=3$. Finally, if $i\in U$, $j\in V$, then $\abs{i\cap j}=1$
and $B\cap C=j\setminus i$, so $\abs{B\cap C}=6$ and we have covered all cases.
\end{proof}

\section{On $2$-transitive symmetric designs}

As another application of the lemma we reprove \cite[Theorem
  1.10]{GrundhoeferMueller} without using character theory. In particular,
Lorimer's and O'Nan's results \cite{O'N85} about the nonexistence of sharply
$2$--transitive sets of permutations in $\PgL_k(q)$ ($k\ge3$) hold by simple
counting arguments.

\begin{theorem} 
Let $G$ be an automorphism group of a nontrivial symmetric design. Then the
stabilizer in $G$ of a point does not contain a subset which is sharply
transitive on the remaining points. In particular, $G$ does not contain a subset
which is sharply $2$-transitive on the points of the design.
\end{theorem}
\begin{proof}
Let $v>k>\lambda$ be the usual parameters of the design. So the set $\Omega'$
of points of the design has size $v$, each block has size $k$, and two distinct
blocks intersect in $\lambda$ point. We will use the easy relation
$(v-1)\lambda=k^2-k$ (see any book on designs).

Fix $\omega\in\Omega'$, let $G_\omega$ be the stabilizer of $\omega$ in $G$,
and suppose that $S\subseteq G_\omega$ is sharply transitive on the set
$\Omega:=\Omega'\setminus\{\omega\}$ of size $v-1$. As each point is contained
in $k<v$ blocks, there is a block $B$ with $\omega\notin B$. Apply Lemma
\ref{lm:doublecount} with $C=B$, so
\[ \sum_{g\in S}\abs{B\cap B^g}=\abs{B}^2=k^2. \]
Let $a$ be the number of $g\in S$ with $B=B^g$. In the remaining
$\abs{S}-a=v-1-a$ cases we have $B\ne B^g$, hence $\abs{B\cap B^g}=\lambda$.

We obtain $ak+(v-1-a)\lambda=k^2$. Recall that $(v-1)\lambda=k^2-k$, so
\[ a(k-\lambda)=k. \]
Now let $B'$ be a block with $\omega\in B'$. Set $B=C=
B'\setminus\{\omega\}$. Then $\abs{B\cap B^g}=k-1$ or $\lambda-1$. Let $b$ be
the frequency of the first case. As above we get
$b(k-1)+(v-1-b)(\lambda-1)=(k-1)^2$, which simplifies to
\[ b(k-\lambda)=v-k. \]
We obtain:
\[ (k-\lambda)^2 \text{ divides } k(v-k), \text{ and } k-\lambda \text{ divides
} k+(v-k)=v. \] On the other hand, the basic relation $(v-1)\lambda=k^2-k$ is
equivalent to $k(v-k)=(v-1)(k-\lambda)$, so $k-\lambda$ divides
$v-1$. Therefore $k-\lambda=1$, hence $k=v-1$ and we have the trivial design,
contrary to our assumption.
\end{proof}

\section{Remarks on $M_{24}$}

In the last section, we sketch a computer based proof showing that
\eqref{eq:01} has an integer solution for $G=M_{24}$ in its permutation
representation on $\Omega^{(2)}$ with $\Omega=\{1,\ldots,24\}$. As the tedious
proofs of Lemmas \ref{L:down}, \ref{L:up} and Theorem \ref{th:intsol} are not
directly related to the main goal of this paper, we omit them and will give
them in a separate paper.

For a subgroup $H\le G$, we consider the following system \eqref{H} of linear 
equations:
\begin{quote} 
Let $\Omega_1,\Omega_2,\dots,\Omega_r$ be the orbits of $H$ on
$\Omega\times\Omega$, and $T$ be a set of representatives for the action of
$H$ on $G$ by conjugation. For $i=1,2,\dots,r$ and $g\in G$ set
\[ 
a_i(g)=\abs{\{(\omega_1,\omega_2)\in\Omega_i|\omega_1^g=\omega_2\}} 
\] 
and consider the system of $r$ linear equations in the variables $x_g$, 
$g\in T$: 
\begin{equation}\label{H}%\tag{H} 
\sum_{g\in T}x_ga_i(g)=\abs{\Omega_i},\;\;i=1,2,\dots,r. 
\end{equation} 
\end{quote} 
The system \eqref{eq:01} is the same as the system \eqref{H} with 
$H=1$. Furthermore, note that $a_i(g)$ depends only on the $H$--class of $g$, 
so the system of equations does not depend on the chosen system $T$ of 
representatives. 
%The following lemma follows directly from the definitions. 
%\begin{lemma} 
%The set $F$ contains a sharply transitive set of permutations if and only if 
%the system (*) has a solution over the rationals for $H=1$ with all 
%$X_g\in\{0,1\}$ for all $g\in T$. 
%\end{lemma} 
 
\begin{lemma}\label{L:down} 
Let $U\le V\le G$ be subgroups of $G$. If \eqref{H} has an integral solution 
for $H=U$, then \eqref{H} has an integral solution for $H=V$. 
\end{lemma} 
\begin{lemma}\label{L:up} 
Let $p^m>1$ be a power of a prime $p$, and $R=\ZZZ/p^m\ZZZ$. Suppose that 
\eqref{H} has a solution in $R$ for some $p'$--subgroup $H$ of $G$. Then also 
\eqref{eq:01} is solvable over $R$. 
\end{lemma} 
The proof of Lemma \ref{L:up} only uses that $\abs{H}$ is a unit in $R$. So if
\eqref{H} has a rational solution for some $H\le G$, then \eqref{eq:01} has a
rational solution too. So the rational solubility of \eqref{eq:01} can be
decided by the rational solubility of \eqref{H} for $H=G$, which gives a very
weak condition.

A useful criterion to decide whether \eqref{eq:01} has an integral solution is
\begin{theorem} \label{th:intsol}
The following are equivalent:\begin{itemize} 
\item[(i)] The system \eqref{eq:01} has an integral solution.
\item[(ii)] For each prime divisor $p$ of $\abs{G}$, the system \eqref{H} has 
  an integral solution for some $p'$--subgroup $H$ of $G$. 
\end{itemize} 
\end{theorem} 
In order to apply this theorem to the action of $G=M_{24}$ on $\Omega^{(2)}$,
we first choose a Sylow $2$--subgroup $H$ of $G$. So $H$ is a $p'$--subgroup
of $G$ for each odd prime $p$. The number of $H$--orbits on $G$ is
$241871$. So the number of unknowns is reduced by a factor
$\abs{G}/241871=1012.2\ldots$. The number of equations is $603$. In order to
solve this system, one can pick about $270$ variables at random, and set the
remaining ones to $0$. Experiments with the computer algebra system Magma
\cite{Magma} show that this system usually has an integral solution.

It remains to take a $2'$--subgroup of $G$. For this we let $H$ be the
normalizer of a Sylow $23$--subgroup. Then $\abs{H}=253$. This reduces the
number of unknowns from $\abs{G}=244823040$ by a factor of about $253$ to
$967692$. Here, picking $520$ unknowns at random usually gives an integral
solution.

In both cases, the running time is a few minutes.

There are several modifications of this method. In \eqref{eq:01} it suffices
to consider the sum over the fixed-point-free elements and $1$, and likewise
in \eqref{H} (and the lemmas and the theorem), it suffices to consider $1$
together with the $H$--orbits on fixed-poin-free elements. However, even under
this assumption, \eqref{eq:01} still has an integral solution. To do so, one
simply sets $x_1=1$ and randomly picks the variables $x_g$ for
fixed-point-free elements $g$ from $T$.

Also, Theorem \ref{th:intsol} and Lemma \ref{L:down} remain true if we replace
`integral' by `non-negative integral'. So we are faced with an integer linear
programming problem. Experiments have shown that \eqref{H} has a
non-negative integral solution for each of the $29$ subgroups $H$ of $G=M_{24}$
with $[G:H]\le26565$.
\bibliographystyle{plain}

\begin{thebibliography}{99}

\bibitem{BCN:graphs}
A.~E. Brouwer, A.~M. Cohen, A.~Neumaier. Distance-regular Graphs,
  vol.~18 of \textit{Ergebnisse der Mathematik und ihrer Grenzgebiete (3)},
  Springer-Verlag, Berlin (1989).

\bibitem{Dembowski} P. Dembowski. Finite geometries.
Springer-Verlag, Berlin-New York, 1968.

\bibitem{FrankDeza} P. Frankl and M. Deza. On the maximum number of permutations
with given maximal or minimal distance, J. Combin. Theory Ser. A, Vol. 22 (1977)
pp. 352-360.

\bibitem{GAP4} {\sc GAP Group.} {GAP --- Groups, Algorithms, and
Programming}. {University of St~Andrews and RWTH Aachen}, {2002},
{Version 4r3}.

\bibitem{Grundhoefer} T. Grundh\"ofer. The groups of projectivities of
finite projective and affine planes. Eleventh British Combinatorial
Conference (London, 1987). Ars Combin. 25 (1988), A, 269--275.

\bibitem{GrundhoeferMueller} T. Grundh\"ofer and P. M\"uller. Sharply
2-transitive sets of permutations and groups of affine projectivities.
\textit{Beitr\"age zur Algebra und Geometrie} 50(1) (2009), 143--154. 

\bibitem{O'N85} M. E. O'Nan. Sharply 2-transitive sets of permutations. In Proc.
Rutgers group theory year, 1983-1984 (New Brunswick, N.J., 1983-1984),
pages 63-67. Cambridge Univ. Press, 1985.

\bibitem{Quistorff} J. Quistorff. A survey on packing and covering problems
in the Hamming permutation space. Electron. J. Combin.  13  (2006),  no. 1,
Article 1, 13 pp. (electronic).

\bibitem{Tarnanen} H. Tarnanen. Upper Bounds on Permutation Codes via Linear
Programming. Europ. J. Combinatorics (1999) 20, 101-114.

\bibitem{Magma} Wieb Bosma, John Cannon, and Catherine Playoust. The Magma
algebra system. I. The user language. J. Symbolic Comput., 24(3-4):235-265,
1997.

\end{thebibliography}

\end{document}